%% file: MQT11.tex
\documentclass[12pt,oneside,openany,article]{memoir}
\usepackage{mempatch}
\nouppercaseheads 
\usepackage[amsthm,thmmarks,hyperref]{ntheorem}
\usepackage{xr-hyper}
\usepackage[noTeX]{mmap}
\usepackage{etex}
\usepackage[english]{babel}
\usepackage[leqno]{mathtools}
\usepackage{mathrsfs}
\usepackage{upgreek}
\usepackage[compress,square,comma,numbers]{natbib}
\usepackage{hypernat} 
\usepackage{sfmath}
\usepackage{microtype} 
\usepackage{hyperxmp}
\usepackage{zref}
\usepackage{nameref}
\usepackage[pdftex,bookmarks,pdfnewwindow,plainpages=false,unicode]{hyperref}
\usepackage{bookmark}
\usepackage{attachfile}
\usepackage{enumitem}
\usepackage{tikz}
\usepackage{amssymb} 

\hypersetup{
colorlinks=true,
linkcolor=black,
citecolor=black,
urlcolor=blue,
pdfauthor={Victor Ivrii},
pdftitle={Global trace asymptotics  in the self-generated magnetic field in the case of Coulomb-like singularities},
pdfsubject={Sharp Spectral Asymptotics},
pdfkeywords={Schrodinger operator, self-generated magnetic field},
bookmarksdepth={4}
}

\numberwithin{equation}{chapter}

\input{defines}
\begin{document}


\title{Global trace asymptotics  in the self-generated magnetic field in the case of Coulomb-like singularities}
\author{Victor Ivrii\thanks{Department of Mathematics, University of Toronto, 40 St. George Street, Toronto, ON, M5S 2E4, Canada, ivrii@math.toronto.edu}}

\maketitle

\chapter{Problem}
\label{sect-1}

Let us consider the following operator (quantum Hamiltonian) in 
$\bR^d$ with $d=3$
\begin{equation}
H=H _{A,V}=\bigl((hD -A)\cdot \boldsigma \bigr) ^2-V(x)
\label{1-1}
\end{equation}
where $A,V$ are real-valued functions and $V$ has a Coulomb-like singularity at $0$ or has several such singularities and is smooth and decays as Coulomb or better at infinity\footnote{\label{foot-1} In \cite{MQT10}  we assumed that $V$  is a smooth function.}. 

Let  $A\in \sH^1$. Then operator $H$ is self-adjoint in $\sL^2(\bR^3,\bC^2)$. We are interested in $\Tr^- H_{A,V}$ (the sum of all negative eigenvalues of this operator). Let
\begin{gather}
\E ^*=\inf_{A\in \sH^1_0(B(0,1))}\E(A),\label{1-2}\\
\E(A)\Def \Bigl( \Tr^-H_{A,V}  +  
\kappa^{-1} h^{-2}\int  |\partial A|^2\,dx\Bigr)
\label{1-3}
\end{gather}
with $\partial A=(\partial_i A_j)$ a matrix. 

This paper is the second step to the recovering sharper asymptotics of the ground state energy for atoms  and molecules   in the self-generated magnetic fields. 

Let $\x_j\in \bR^3$ ($j=1,\ldots,M$, where $M$ is fixed) be singularities (``nuclei''). We assume that
\begin{gather}
V=\sum_{1\le j\le M} \frac{z_j} {\ell_j(x)} + W(x)
\label{1-4}\\
\shortintertext{where $\ell_j(x)=\frac{1}{2}|x-\x_j|$,}
z_j\ge 0, \ z_1+\ldots +z_M \asymp 1,\label{1-5}\\[3pt]
|D^\alpha W|\le C_\alpha \sum_{1\le j\le M}  z_j\bigl(\ell_j(x)+1\bigr)^{-1}\bigl(\ell_j(x)\bigr)^{-|\alpha|}\qquad\forall \alpha:|\alpha|\le 2
\label{1-6}
\end{gather}
but at first stages we will use some weaker assumptions. Later we assume that $V(x)$ decays at infinity sufficiently fast.

In this paper we assume that $\kappa \in (0,\kappa^*]$ where $0<\kappa^*$ is a small constant. As $\kappa=0$ we set $A=0$ and consider $\Tr^- H_{A,V}$; then our results will not be new.

\chapter{Estimates of the minimizer}
\label{sect-2}

Let us consider a Hamiltonian with potential $V$ and let $A$ be a   minimizing expression (\ref{1-3}) magnetic field. We say that $A$ is a \emph{minimizer\/} and in the framework of our problems we will prove it existence.

\section{Preliminary analysis}
\label{sect-2-1} 

First we start from the roughest possible estimate:
\begin{proposition}\label{prop-2-1}
Let $\kappa \le \kappa^*$. Then the near-minimizer $A$ satisfies
\begin{gather}
|\int \bigl(\tr e_1(x,x,0)-\Weyl_1(x)\bigr)\,dx|\le Ch^{-2}\label{2-1}\\
\shortintertext{and}
\|\partial A\| \le C\kappa^{\frac{1}{2}}
\label{2-2}
\end{gather}
\end{proposition}

\begin{proof}
Definitely (\ref{2-1})--(\ref{2-2}) follow from the results of \cite{EFS3} but we give an independent easier proof based on \cite{MQT10}. 

\begin{enumerate}[fullwidth, label=(\roman*)]
\item \label{proof-2-1-i}
First, let us pick up  $A=0$ and consider $\Tr \bigl(\psi_\ell E(0)\psi_\ell\bigl)$ with  cut-offs $\psi_\ell(x)=\psi((x-\x_j)/\ell)$ where $\psi\in \sC_0^\infty (B(0,1))$ and equals $1$ in $B(0,\frac{1}{2})$. Here and below 
$E (\tau)= \uptheta (\tau -H_{A,V})$ is a spectral projector of $H$.

Then 
\begin{equation}
|\Tr \bigl(\psi_\ell  H_{A,V} ^-(0)\psi_\ell\bigr)|\le Ch^{-2}\qquad
\text{as\ \ } \ell=\ell_*\Def h^2.
\label{2-3}
\end{equation}
On the other hand, contribution of $B(x,\ell )$ with 
$\ell(x)=\frac{1}{2} \min_j |x-\x_j|\ge \ell_*$ to the Weyl error does not exceed  $C\rho ^2\hbar^{-1}=C\rho^3 \ell h^{-1}$ where $\hbar = h/\rho \ell$ in the rescaling; so  after summation over $\ell\ge \ell_*$ we get $O(h^{-2})$ provided $\rho^2\le C\ell^{-1}$. Therefore we arrive to the following rather easy inequality:
\begin{equation}
|\int \bigl(\tr e_1(x,x,0)-\Weyl_1(x)\bigr)\,dx |\le Ch^{-2} .
\label{2-4}
\end{equation}
 This is what rescaling method gives us without careful study of singularity. 

\item \label{proof-2-1-ii}
On the other hand, consider $A\ne 0$. Let us prove first that
\begin{equation}
\Tr^- (\psi_\ell H \psi_\ell) \ge Ch^{-2} -Ch^{-2}\int |\partial A|^2\,dx
\qquad\text{as\ \ } \ell=\ell_*.
\label{2-5}
\end{equation}
Rescaling $x\mapsto x/\ell$ and  $\tau \mapsto \tau/\ell$ and therefore $h\mapsto h \ell^{-\frac{1}{2}}\asymp 1$ and $A\mapsto A\ell^{\frac{1}{2}}$  (because singularity is Coulomb-like),  we arrive to the same  problem with the same $\kappa$ (in contrast to section~\ref{MQT10-sect-4} of \cite{MQT10} where $\kappa\mapsto \kappa \ell$ because of different scale in $\tau$ and $h$) and with $\ell=h=1$.

However this estimate follows from the proof in section \ref{ES3-sec:tooth} of \cite{ES3} of Lemma~\ref{ES3-lm:h1}, namely from (\ref{ES3-running})--(\ref{ES3-secondclr}) with $Z=d=1$.

\item \label{proof-2-1-iii}
Consider now $\psi_\ell$  as in \ref{proof-2-1-i} with $\ell\ge \ell_*$. Then according to theorem~\ref{MQT10-thm-4-1} of \cite{MQT10} rescaled
\begin{equation}
\Tr^- \bigl(\psi_\ell H_{A,V} \psi_\ell\bigr) \ge -C \rho^3 \ell h^{-1} - 
Ch^{-2} \int _{B(x, 2\ell/3)} |\partial A|^2\,dx.
\label{2-6}
\end{equation}
Really, rescaling of the first part is a standard one and in the second part we should have  in the front of the integral a coefficient 
$\kappa^{-1}h^{-2}\rho^2 \times \rho^{-2} \ell  (h/\rho \ell)^{-2}$ where factor $\rho^2$ comes from the scaling of the spectral parameter, factor $\rho^{-2}$ comes from the scaling of the magnitude of $A$, factor $\ell=\ell^3 \times \ell^{-2}$ comes from the scaling of $dx$ and $\partial$ respectively, and $h/(\rho \ell)$ is a semiclassical parameter after rescaling. So, we acquire a factor 
$\rho^2 \ell \le C$.

Then 
\begin{equation}
\int \bigl(\tr e_1(x,x,0)-\Weyl_1(x)\bigr)\,dx \ge 
-Ch^{-2}- Ch^{-2} \int |\partial A|^2\,dx
\label{2-7}
\end{equation}
and adding magnetic field energy we find out that the left-hand expression of (\ref{2-1}) is greater than the same expression with $A=0$ plus 
$(C-\kappa^{-1})h^{-2}\|\partial A\|^2$ minus $Ch^{-2}$ which implies (\ref{2-1}) and (\ref{2-2}) as $A$ is supposed to be a near-minimizer.
\end{enumerate}
\end{proof}

\begin{remark}\label{rem-2-2}
We are a bit ambivalent about convergence of $\int \Weyl_1(x)\, dx$ at infinity, as for Coulomb potential it diverges. In this case however we can \underline{either} replace $H_{A,V}$ by $H_{A,V}+\eta$ with a small parameter $\eta>0$ \underline{or} consider the left-hand expression of (\ref{2-1}) plus magnetic field energy as an object to minimize.
\end{remark}

\section{Rough estimate to a minimizer. I}
\label{sect-2-2}

Let us repeat arguments of subsection~\ref{MQT10-sect-1-3} of \cite{MQT10}. Let us consider equation for an minimizer $A$ as in (\ref{MQT10-1-13}) of \cite{MQT10}:  
\begin{equation}
\Delta A= -2\kappa h^2 \sum_k \bigl(\upsigma_j\upsigma_k ( hD_k-A_k)_x  + 
\upsigma_k\upsigma_j (hD_k-A_k)_y  \bigr) e (x,y,\tau) |_{y=x}
\label{2-8}
\end{equation}
If we scale with the scale $x\mapsto x/\ell$, $\tau\mapsto \tau/\rho^2$, $h\mapsto \hbar=h/(\rho\ell)$ then (\ref{2-8}) would become
\begin{multline}
\Delta A=\\
-2\kappa \rho^3\ell  \hbar^2 \sum_k 
\bigl(\upsigma_j\upsigma_k  ( \hbar D_k-\rho^{-1} A_k)_x  + 
\upsigma_k\upsigma_j (\hbar D_k-\rho^{-1}A_k)_y  \bigr) e (x,y,\tau) |_{y=x}
\label{2-9}
\end{multline}
and since so far $\rho^2\ell=1$ we arrive to
\begin{multline}
\Delta A=\\
-2\kappa \rho   \hbar^2 \sum_k 
\bigl(\upsigma_j\upsigma_k  ( \hbar D_k-\rho^{-1} A_k)_x  + 
\upsigma_k\upsigma_j (\hbar D_k-\rho^{-1}A_k)_y  \bigr) e (x,y,\tau) |_{y=x}.
\label{2-10}
\end{multline}

\begin{enumerate}[label=(\roman*), fullwidth]
\item Plugging for $u=E (0)f$ and repeating arguments of \cite{MQT10} we conclude that in the rescaled coordinates 
\begin{multline}
\|\hbar D_x u \|\le  
\|((\hbar D_x -\rho^{-1}A_x)\cdot\boldsigma)u\|+C\rho^{-1}\|A\|_6 \|u\|_3 \le \\[3pt]
\|((\hbar D_x -\rho^{-1}A_x)\cdot\boldsigma)u\|+
C\rho^{-1}\hbar ^{-\frac{1}{4}}\|A\|_6 \|u\|^{\frac{3}{4}} 
\cdot  \|\hbar D_x u\|^{\frac{1}{4}}\le\\[3pt]
\|((\hbar D_x -\rho^{-1}A_x)\cdot\boldsigma)u\|+ \frac{1}{2} \|\hbar  D_x u\|
+ C (\rho^{-1}\hbar^{-\frac{1}{4}}\|A\|'_6)^{\frac{4}{3}}  \|u\|
\label{2-11}
\end{multline}
where $\|A\|_6$ calculated in the rescaled coordinates is $\ell^{-3/6}\|A\|_{6,\orig}$ (where subscript ``$\mathsf{orig}$''means that the norm is calculated in the original coordinates) which does not exceed 
$C\ell^{-\frac{1}{2}}\|\partial A\|_{\orig}\le C\ell^{-\frac{1}{2}}\kappa^{\frac{1}{2}}$ due to (\ref{2-2}) and therefore 
\begin{equation}
\|\hbar D_x u \| \le C\bigl( 1+ \rho ^{-\frac{3}{4}}h^{-\frac{1}{4}}\ell^{-\frac{1}{4}}\kappa^{\frac{1}{2}} \bigr)^{\frac{4}{3}}\|f\|.
\label{2-12}
\end{equation}
Continuing arguments of section~\ref{MQT10-sect-1-3} of\cite{MQT10} we conclude that in the rescaled coordinates
\begin{multline}
\hbar \|\Delta \partial A\|_{\infty, B(x,1)} + 
\|\Delta A\|_{\infty, B(x,1)} \le  K\ell ^{-\frac{1}{2}} \\
 K \Def C\kappa   \hbar^{-1}
\bigl(1+ \rho ^{-1}\hbar^{-\frac{1}{4}}\ell^{-\frac{1}{2}}\kappa^{\frac{1}{2}} \bigr)^4  .
\label{2-13}
\end{multline}
Then \underline{either}
\begin{gather}
\|\partial   A\|_{\infty, B(x,\frac{3}{4})} +
\|\partial  A \|^*_{\infty, B(x,\frac{3}{4})} \le C K \ell^{-\frac{1}{2}} \label{2-14}\\
\shortintertext{\underline{or}}
\|\partial  A\|_{\infty, B(x,\frac{3}{4})} + 
\|\partial  A \|^*_{\infty, B(x,\frac{3}{4})} \le 
C \|\partial A\|= C\|\partial A\|_{\orig} \ell^{-\frac{1}{2}}\le  
C\kappa^{\frac{1}{2}} \ell^{-\frac{1}{2}}.\label{2-15}
\end{gather}
where in the rescaled coordinates 
\begin{equation}
\|B\|^*\Def\sup _{x,y} |B(x)-B(y)|\cdot |x-y|^{-1}(1+|\log |x-y||)^{-1}.
\label{2-16}
\end{equation}

In the latter case (\ref{2-15}) we have in the original coordinates
\begin{equation}
\|\partial  A\|_{\infty, B(x,\frac{3}{4}\ell)}\le 
C\kappa^{\frac{1}{2}} \ell^{-\frac{3}{2}}
\label{2-17}
\end{equation}
and we are rather happy because then the effective intensity of the magnetic field in $B(x,\ell)$ is 
$\rho^{-1}\ell \|\partial  A\|_{\infty, B(x,\ell)} \le C\kappa^{\frac{1}{2}}$ if we take $\rho =\ell^{-\frac{1}{2}}$.

In the former case (\ref{2-14}) let us consider (still in the rescaled coordinates) $\beta(x)=|\partial  A (x)|\ell^{\frac{1}{2}}$.  Then   $\beta(x)$ has the same magnitude $\beta(y)$ in $\gamma$-vicinity of $y$ with 
$\gamma= \epsilon \beta(y) K^{-1}|\log (\beta(y) K^{-1})|^{-1}$ (or $\gamma_1=\epsilon$, whatever is smaller). But then in the rescaled coordinates
\begin{gather*}
\ell^{-1} \beta^2 (\beta K^{-1}|\log (\beta K^{-1})|^{-1})^3\le 
C \|\partial A\|^2 \le C \|\partial A\|^2 _{\orig}\ell ^{-1}\le 
C \kappa^{\frac{1}{2}}\ell^{-1}\\
\shortintertext{and then}
\beta \le C\kappa^{\frac{1}{10}} K^{\frac{3}{5}}
|\log (\beta K^{-1})|^{\frac{3}{5}}
\end{gather*}
which implies
\begin{equation}
\beta  \le C  \hbar ^{-\frac{6}{5}}|\log \hbar|^{\frac{3}{5}},\qquad \hbar=h\ell^{-\frac{1}{2}}
\label{2-18}
\end{equation}
(as $\gamma \asymp 1$ the same arguments lead us to (\ref{2-17})).

Therefore in the first round of our estimates we arrive to the estimates in the rescaled coordinates
\begin{equation}
|\partial A| \le   \beta\ell^{-\frac{1}{2}},\qquad 
\beta \Def C \hbar ^{-\frac{6}{5}-\delta}
\label{2-19}
\end{equation}
where we just estimated $|\log \hbar|$ by $\hbar^{-\delta_1}$; below we increase $\delta$ if needed but it still remains an arbitrarily small exponent.

\item
In the second round we do not invoke $\|A\|_6$ but rather 
$\|A\|_{\infty, B(y,\gamma)}\le C \beta \ell^{-\frac{1}{2}}\gamma$  where we consider a ball of radius $\gamma \le 1$ in the rescaled coordinates (and subtract a constant from $A$ if needed), resulting in
\begin{equation*}
\|\Delta A\|_{\infty, B(x,\gamma)} \le 
C\kappa \hbar^{-1}\rho  
\bigl(1+ \beta \gamma \ell^{-\frac{1}{2}} \rho ^{-1} \bigr)^4.
\end{equation*}
Let us increase $\rho$ to 
$\rho'= (\beta h \ell^{-\frac{3}{2}})^{\frac{1}{2}}=  
C\ell^{-\frac{1}{2}}  (\beta \hbar)^{\frac{1}{2}} = C\ell^{-\frac{1}{2}} 
\hbar ^{-\frac{1}{10}-\delta}\ge \ell^{-\frac{1}{2}}$
and use 
$\gamma = h /\rho'\ell =  \hbar^{\frac{11}{10}-\delta}\le 1$. Then we arrive to 
\begin{equation}
\|\Delta A\|_{\infty, B(x,\gamma)} \le K\ell^{-\frac{1}{2}}, \qquad 
K\Def  C\kappa \hbar^{-\frac{7}{5}-\delta}.
\label{2-20}
\end{equation}
Repeating arguments of the first rounds we conclude that \underline{either} (\ref{2-17}) holds \underline{or}
\begin{equation}
|\partial A| \le \beta\ell^{-\frac{1}{2}}, \qquad 
\beta  \Def K ^{\frac{3}{5}-\delta}  \le   C\hbar^{-\frac{21}{25}-\delta};
\label{2-21}
\end{equation}
then  rescaled magnetic field is $O(\beta\ell^{-\frac{1}{2}} /\rho) =O(\hbar^{-21/25-\delta})$. 
Here we returned to the natural scale $(\ell,\rho)$ with $\rho=\ell^{-\frac{1}{2}}$.

\item
One can also run third etc rounds, using partially arguments of subsection~\ref{MQT10-sect-2-1} of \cite{MQT10}; then the rescaled magnetic field is $O(\hbar^{-\delta})$. However to  prove that  the rescaled magnetic field  $O(1)$  we need to modify them, and we do it in the next subsection.
\end{enumerate}

\section{Rough estimate. II}
\label{sect-2-3}

In this step we repeat arguments of subsection~\ref{MQT10-sect-2-1} of \cite{MQT10} but we have a problem: we cannot use $\mu =\|\partial A\|_\infty$ as we have  domains 
$\cX_r= \{x: \ell(x)\ge r\}$ rather than the whole space. So we get the following analogue of (\ref{MQT10-2-19}) of \cite{MQT10} in the rescaled coordinates:
\begin{multline}
\|\Delta A\|_{\infty, B(x, \frac{3}{4} )} + 
\hbar \|\Delta \partial A\|_{\infty, \frac{3}{4} } \le \\
C \kappa \rho \Bigl( \bar{\mu}  + 
\bar{\mu} ^{-1} \hbar^{\frac{1}{2}(\theta-1)}\rho^{-\frac{1}{2}}
 \|\partial A\|_{\sC^\theta (B(x,1))}^{\frac{1}{2}}\Bigr)
\label{2-22}
\end{multline}
with $\bar{\mu}=\max(\mu,1)$, and 
$\mu = \rho^{-1} |\partial A|_{\infty, B(x,1)}$. But then 
\begin{equation}
\hbar^{\theta-1}\rho^{-1} \|\partial A\|_{\sC^\theta (B,(x,\frac{1}{2}))} \le  
\epsilon  \hbar^{(\theta-1)}\rho^{-1}
 \|\partial A\|_{\sC^\theta (B(x,1))} + 
 C\kappa \mu +C\kappa.
 \label{2-23}
\end{equation}
Obviously in the right-hand expression we can replace 
$\mu = \rho^{-1} |\partial A|_{\infty, B(x,1)}$ by any other norm, in particular by $\sL^2$-norm 
\begin{equation*}
\mu = \rho^{-1} \|\partial A\|_{ B(x,1)}= \rho^{-1}\ell^{-\frac{1}{2}} \|\partial A\|_{ \orig}
\end{equation*}
which would be less than $C\kappa^{\frac{1}{2}}$. 

Let $\nu (r)=\sup_{x:\,\ell(x)\ge r} f(x)$ where $f(x)$ is the left-hand expression of (\ref{2-12}) calculated for given $x$ in the rescaled coordinates. Then (\ref{2-23}) implies that 
\begin{gather*}
\nu (r)\le \frac{1}{2}\nu (\frac{1}{2} r) +C\kappa^{\frac{1}{2}}\\
\shortintertext{which in turn implies that}
\nu (r) \le \frac{1}{2}\nu (2^{-n}r) + 2C,\qquad n\ge 1,\\
\shortintertext{and therefore}
\nu(r)\le 4C\kappa^{\frac{1}{2}}+ 4 \sup_{C_0 h^2\le \ell(x)\le 2C_0h^2}  f(x) \le C_1\kappa^{\frac{1}{2}}
\end{gather*}
due to the rough estimate (because $\hbar\asymp 1$  as $\ell(x)\asymp h^2$). Then going to the original coordinates we arrive to estimates  below:

\begin{proposition}\label{prop-2-3}
Let $\kappa\le \kappa^*$, $\rho=c\ell^{-\frac{1}{2}}$. Let $A$ be a minimizer. Then   for $\ell(x)\ge \ell_*=h^2$ \textup{(\ref{2-17})} holds and also
\begin{gather}
|\partial^2 A (x)-\partial^2 A(y)|\le C\kappa^{\frac{1}{2}}\ell^{-\frac{5}{2}}|x-y|^{\theta} \ell^{\theta/2}\ell_*^{-\theta/2}\qquad 0<\theta <1,
\label{2-24}\\
\shortintertext{and}
|\partial A (x)-\partial A(y)|\le C\kappa^{\frac{1}{2}}\ell^{-\frac{5}{2}}|x-y| (1+|\log |x-y||)  .
\label{2-25}
\end{gather}
\end{proposition}

\begin{remark}\label{rem-2-4}
\begin{enumerate}[label=(\roman*), fullwidth]
\item So far we used only assumption that
\begin{equation}
|\partial ^\alpha V|\le C\rho^2 \ell^{-|\alpha|}\qquad \forall 
\alpha :|\alpha|\le 2
\label{2-26}
\end{equation}
with $\rho =\ell^{-\frac{1}{2}}$ but even this was excessive. 

\item In this framework however we cannot prove better estimates as (\ref{2-17}) always remains a valid alternative even if $\rho  \ll \ell^{-\frac{1}{2}}$.

\item Originally we need an assumption (\ref{MQT10-2-4}) of \cite{MQT10} $|V|\ge \epsilon_0$, but for $d\ge 3$  one can easily get rid off it  by rescaling technique; see also corollary~\ref{MQT10-cor-2-3}(ii).
\end{enumerate}
\end{remark}

Consider now zone $\{x:\ \ell(x) \le \ell_*\}$. 

\begin{proposition}\label{prop-2-5}
Let $\kappa\le \kappa^*$, $\rho\le c\ell^{-\frac{1}{2}}$. Let $A$ be a minimizer.   Then $|\partial A|\le C h^{-3}$ as $\ell(x) \le \ell_*$.
\end{proposition}

\begin{proof}
Proof is standard, based on rescaling (then $\hbar=1$) and equation (\ref{2-8}) for $A$. We leave details to the reader.
\end{proof}

Let us slightly improve estimate to $A$. We already know that 
$|\partial A (x)|\le C_0\beta $ with $\beta =\ell^{-\frac{3}{2}}$ and using a standard rescaling technique we conclude that 
\begin{equation}
|\Delta A|\le C\kappa \rho^2 \beta + C\kappa \rho^3 \ell^{-1}
\label{2-27}
\end{equation}
which does not exceed $C\kappa \ell^{-\frac{5}{2}}$ which implies 

\begin{proposition}\label{prop-2-6}
In our framework
\begin{enumerate}[label=(\roman*), fullwidth]
\item As $\ell (x)\ge h^2$
\begin{gather}
|A|\le C\kappa \ell^{-\frac{1}{2}},\qquad
|\partial A|\le C\kappa \ell ^{-\frac{3}{2}},\label{2-28}\\
|\partial A (x)-\partial A(y)|\le C_\theta\kappa \ell^{-\frac{3}{2}-\theta}|x-y|^\theta \qquad \text{as\ \ } |x-y|\le \frac{1}{2}\ell (x)
\label{2-29}
\end{gather}
for any $\theta \in (0,1)$
\item as $\ell (x)\le h^2$  these estimates hold with $\ell (x)$ replaced by $h^2$.
\end{enumerate} 
\end{proposition}

Here in comparison with old estimates we replaced factor $\kappa^{\frac{1}{2}}$ by $\kappa$ which is an advantage.

Consider now zone $\{\ell \ge \max (a,1)\}$ and assume that 
\begin{equation}
\rho \le C\ell^{-\nu}\qquad \text{as\ \ } \ell \ge \frac{1}{2} 
\label{2-30}
\end{equation}
with $\nu >1$. Then if also $\beta =O(\ell^{-\nu_1})$ as $\ell\ge 1$ the right hand expression of (\ref{2-27}) does not exceed 
$C\kappa (\ell^{-3\nu-1}+\ell^{-\nu_1-2\nu})$ and therefore we almost upgrade estimate to $\beta$ to $O(\ell^{-3\nu}+\ell^{-\nu_1-2\nu+1})$ and repeating these arguments sufficiently many times to $O(\ell^{-3\nu})$. However, there are obstacles: first, as $\nu >1$ we get 
\begin{equation*}
A_j = \sum_m \upalpha_{j,m} |x-\x_m|^{-1} + O(\ell^{-1-\delta})
\end{equation*}
with constant $\upalpha_{j,m}$; however assumption $\nabla\cdot A=0$ implies $\upalpha_{j,m}=0$ and we pass this obstacle. The second obstacle 
\begin{equation*}
A_j = \sum_{k,m} \upalpha_{jk,m} (x_k-\x_{k,m})|x-\x_m|^{-3} +
O(\ell^{-2})
\end{equation*}
with constant $\upalpha_{jk,m}$ we cannot pass as assumption $\nabla\cdot A=0$ implies only that modulo gradient 
$A=\sum_m \upbeta_m\times \nabla \ell_m^{-1}$ with constant vectors $\upbeta_m$ and one cannot pass this obstacle.

Therefore we upgrade (\ref{2-28})--(\ref{2-29}) there:

\begin{proposition}\label{prop-2-7}
In our framework assume additionally that  \textup{(\ref{2-30})} holds. Then as $\nu>\frac{4}{3}$
\begin{gather}
|A|\le C\kappa \ell^{-2},\qquad
|\partial A|\le C\kappa \ell ^{-3},\label{2-31}\\
|\partial A (x)-\partial A(y)|\le C_\theta\kappa \ell^{-3-\theta}|x-y|^\theta \qquad \text{as\ \ } |x-y|\le \frac{1}{2}\ell (x)
\label{2-32}
\end{gather}
as $\ell(x)\ge 1$ (for all $\theta \in (0,1)$).
\end{proposition}

\begin{remark}\label{rem-2-8}
\begin{enumerate}[label=(\roman*), fullwidth]
\item In application we are interested in $\nu=2$;

\item We cannot improve (\ref{2-31})--(\ref{2-32}) no matter how fast $\rho$ decays.
\end{enumerate}
\end{remark}

\chapter{Tauberian theory}
\label{sect-3}

Recall that the standard Tauberian theory results in the remainder estimate $O(h^{-2})$. Really, as the rescaled magnetic field intensity is no more than $C\kappa^{\frac{1}{2}}$, contribution of $B(x,\ell(x))$ to the Tauberian error does not exceed $C\rho^2 \times \hbar^{-1}= C\rho^3 \ell h^{-1}$ which as 
$\rho \asymp \ell^{-\frac{1}{2}}$ translates into $C\ell^{-\frac{1}{2}}h^{-1}$ and summation over $\{x: \ell(x)\ge \ell_*=h^2\}$ results in $Ch^{-2}$. On the other hand, contribution of $\{x: \ell(x)\le \ell_*=h^2\}$ into asymptotics does not exceed $C\hbar^{-3}\ell_*^{-1}= Ch^{-2}$ as $\hbar=1$.

However now we can unleash  arguments of \cite{ivrii:ground}.  Recall that we are looking at
\begin{equation}
\Tr (\psi H^-_{A,V}\psi)=  \Tr (\phi_1 H^-_{A,V}\phi_1) +
\Tr (\phi_2 H^-_{A,V}\phi_2)
\label{3-1}
\end{equation}
where $\psi^2=\phi_1^2+\phi_2^2$, $\supp \phi_1\subset \{x, |x|\le 2R\}$,
$\supp \phi_2\subset \{x, R\le |x|\le a\}$  and we compare it with the same expression calculated for $H_{A,V^0}$ with $V^0=z|x|^{-1}$. Here we assume that \begin{gather}
a \le 1, \qquad z\asymp 1
\label{3-2}\\
\shortintertext{and} 
|D^\alpha (V-V^0)|\le c_0 a^{-1} \ell^{-|\alpha|}\qquad \forall \alpha:|\alpha|\le 3.
\label{3-3}
\end{gather}
The latter assumption is too restrictive and could be weaken. Then 
\begin{multline}
\Tr \bigl(\phi_2 (H^-_{A,V} - H^-_{A,V^0})\phi_2\bigr)= \\
\int \bigl(\Weyl_1(x) -\Weyl_1^0(x)\bigr)\,\phi^2_2(x) \,dx + O(R^{-\frac{1}{2}}h^{-1})
\label{3-4}
\end{multline}
where $\Weyl^0_1$ and $\Weyl^0$ are calculated for operator with potential $V^0$. Really, we prove this for each operator $H_{A,V}$ and $H_{A,V^0}$ separately\footnote{\label{foot-2} Sure, such formula requires two-term expression but one can verify easily that the second term is $0$.}.

On the other hand, considering $V^\zeta= V^0 (1-\zeta) +V \zeta = V^0 + W\zeta$ and following \cite{ivrii:ground} we can rewrite the similar expression albeit for $\phi_2=1$ as 
\begin{equation}
\Tr \int_0^1 W\uptheta (-H_{A,V^\zeta})\, d\zeta
\label{3-5}
\end{equation}
and applying the semiclassical approximation (under temporary assumption that $W$ is supported in $\{x:\ |x|\le 4 R\}$) one can prove that as $\phi_1=1$
\begin{multline}
\Tr \bigl(\phi_1 (H^-_{A,V} - H^-_{A,V^0})\phi_1\bigr)= \\
\int \bigl(\Weyl_1(x) -\Weyl_1^0(x)\bigr)\,\phi^2_1(x) \,dx + O(a^{-1}Rh^{-2}).
\label{3-6}
\end{multline}
Really, contribution of ball $B(x,\ell(x))$ does not exceed 
$Ca ^{-1} \hbar^{-2}= Ca^{-1} \ell (x)h^{-2}$ and summation with respect to partition as $\ell (x)\le R$ returns $Ca^{-1}Rh^{-2})$; meanwhile contribution of $\{x: \ell(x)\le \ell_*\}$ does not exceed $Ca^{-1}\hbar^{-2}=Ca^{-1}$ as there $\hbar=1$.

One can get rid off the temporary assumption and take $\phi_1$ supported in $\{x:\ell(x)\le 2R\}$ instead. 

Therefore we arrive to
\begin{proposition}\label{prop-3-1}
Under assumption \textup{(\ref{3-3})}
\begin{multline} 
\Tr \bigl(\psi (H^-_{A,V} - H^-_{A,V^0})\psi\bigr)= \\
\int \bigl(\Weyl_1(x) -\Weyl_1^0(x)\bigr)\,\psi^2(x) \,dx + O\bigl(a^{-\frac{1}{3}}h^{-\frac{4}{3}}\bigr)
\label{3-7}
\end{multline}
\end{proposition}

Really,  $a^{-\frac{1}{3}}h^{-\frac{4}{3}}$ is 
$R^{-\frac{1}{2}}h^{-1}+ a^{-1}Rh^{-2}$ optimized by 
$R\asymp R_*\Def (ah)^{\frac{2}{3}}$; as $h^2 \le a$ we note that 
$h^2\le R_* \le a$. 

\begin{corollary}\label{cor-3-2}
\begin{enumerate}[label=(\roman*), fullwidth]
\item As $M=1$ equality \textup{(\ref{3-7})} remains valid with $\psi=1$ and 
$a= 1$.
\item As $M\ge 2$ and $a\ge h^2$ equality  \textup{(\ref{3-7})} becomes
\begin{multline} 
\Tr \bigl(\psi (H^-_{A,V} - H^-_{A,V^0})\psi\bigr)= \\
\int \bigl(\Weyl_1(x) -\Weyl_1^0(x)\bigr)\,\psi^2(x) \,dx + O\bigl((a^{-\frac{1}{3}}+1)h^{-\frac{4}{3}}\bigr)
\label{3-8}
\end{multline}
where we reset case $a\ge 1$ to $a=1$.
\end{enumerate}
\end{corollary}

\begin{remark}\label{rem-3-3}
One can apply much more advanced arguments of \cite{ivrii:MQT1} or section~\ref{book_new-sect-12-5} of \cite{futurebook}. Unfortunately using these arguments so far I was not able to improve the above results unless $\kappa\ll 1$. More precisely, I proved estimate $O(\kappa h^{-\delta-\frac{4}{3}} + h^{-1})$ as $a=1$ (or even $o(h^{-1})$ as $a\gg 1$, $\kappa=o(h^{\frac{1}{3}+\delta})$ and some assumptions of global nature are fulfilled). However as I still hope to improve these results, I am not including them here.
\end{remark}

\chapter{Single singularity}
\label{sect-4}

\section{Coulomb potential}
\label{sect-4-1}

Consider now exactly Coulomb potential: $V=z |x|^{-1}$. Then according to Theorem~\ref{EFS3-thm:scott} of \cite{EFS3} as $h=1$, $z=1$ and 
$0< \kappa \le \kappa^*$
\begin{multline}
\lim_{R\to\infty} \biggl( \inf_A \Bigl( \Tr^- \bigl( \phi_R H_A \phi_R\bigr) + \frac{1}{\kappa} \int |\partial A|^2 \,dx
\Bigr)\\
- \int  \Weyl_1 (x)\phi_R^2(x)\,dx \biggr) =: 2 z^2 S(z\kappa).
\label{4-1}
\end{multline}
which according to Lemma~\ref{EFS3-S=S} of \cite{EFS3} coincides with 
\begin{multline}
   \lim_{\eta\to 0^+} \biggl( \inf_A \Bigl( \Tr^- \bigl(  H_A  +\eta\bigr) + \frac{1}{\kappa}  \int |\partial A|^2 \,dx
\Bigr) \\
 -  \int \Weyl_1(H_A+\eta,x)\,dx \biggr) = 2z^2S(z\kappa).
 \label{4-2}
\end{multline}
Here $\phi \in \sC_0^\infty (B(0,1))$, $\phi=1$ in $B(0,\frac{1}{2})$, $\phi_R=\phi(x/R)$. Also due to scaling for $z>0$ one has a Scott coefficient $2z^2 S(\kappa z)$.

\begin{proposition}\label{prop-4-1}
 As $0<\kappa <\kappa'$
 \begin{equation}
 S(\kappa')\le S(\kappa) \le S(\kappa') + C \kappa '(\kappa^{-1}-\kappa'^{-1}).
 \label{4-3}
 \end{equation}
\end{proposition}

\begin{proof}
Monotonicity of $S(\kappa)$ is obvious.

Let $0<\kappa <\kappa' <\kappa''\le \kappa^*$. Then for any $\varepsilon>0$ if $R=R_\varepsilon$ is large enough    then the left-hand expression in (\ref{4-1}) for $\kappa'$ (without $\inf$ and $\lim$) is greater than $S(\kappa'') -\varepsilon + (\kappa'^{-1}-\kappa''^{-1})\|\partial A\|^2$;   also, if $A$ is an almost minimizer there, it is less than $S(\kappa')+\varepsilon$. 

Therefore 
$(\kappa'^{-1}-\kappa''^{-1})\|\partial A\|^2\le |S(\kappa'') -S(\kappa') | +2\varepsilon$. But then 
\begin{multline*}
S(\kappa) -\varepsilon \le S(\kappa')+\varepsilon + (\kappa^{-1}-\kappa'^{-1})\|\partial A\|^2\le \\
S(\kappa')+\varepsilon + C (\kappa^{-1}-\kappa'^{-1})(\kappa'^{-1}-\kappa''^{-1})^{-1}\bigl(|S(\kappa'') -S(\kappa') | +2\varepsilon\bigr)
\end{multline*}
and therefore
\begin{equation}
(\kappa^{-1}-\kappa'^{-1}) ^{-1} |S(\kappa) - S(\kappa') |\le 
(\kappa'^{-1}-\kappa''^{-1})^{-1}|S(\kappa') -S(\kappa'') | 
\label{4-4}
\end{equation}
which for $\kappa''=\kappa^*$ implies (\ref{4-3}).
\end{proof}

\begin{remark}\label{rem-4-2}
Using global equation (\ref{2-8}) we conclude that as 
\begin{gather}
|\partial^\alpha  A| \le C\kappa \ell^{-1-|\alpha|}\qquad \ell\ge 1,\ |\alpha|\le 1,\label{4-5}\\
|\partial^\alpha  A| \le C\kappa \ell^{-\frac{1}{2}-|\alpha|}\qquad \ell\le 1,\ |\alpha|\le 1,\label{4-6}\\
\|\partial A\|^2 \le C\kappa^2.\label{4-7}
\shortintertext{Then}
 S'(\kappa) \le C, \qquad |S(\kappa (1+\eta))-S(\eta)|\le C\kappa \eta.\label{4-8}
\end{gather}
\end{remark}

\section{Main theorem}
\label{sect-4-2}

In the ``atomic'' case $M=1$ we arrive instantly to

\begin{theorem}\label{thm-4-3}
If $M=1$, $\kappa \le \kappa^*$ then
\begin{equation}
\E^* = \int \Weyl_1(x)\,dx + 2 z^2 S(z \kappa) h^{-2}+ O(h^{-\frac{4}{3}}).
\label{4-9}
\end{equation}
\end{theorem}

\begin{proof}
If $A$ satisfies minimizer properties  then in virtue of corollary~\ref{cor-3-2} 
\begin{multline}
\Tr^-  H_{A,V} -\int \Weyl_1 (x)\,dx \equiv  
\Tr^-  H_{A,V^0} -\int \Weyl^0_1 (x)\,dx \\
\mod O(h^{-\frac{4}{3}})
\label{4-10}
\end{multline}
and adding magnetic energy and plugging either minimizer for $V$ or for $V^0$ we get
\begin{align}
&\inf_A \Bigl(\Tr^-  H_{A,V} -\int \Weyl_1 (x)\,dx +
\frac{1}{\kappa h^2}\int |\partial A|^2\,dx\Bigr)
\lesseqgtr  \label{4-11}\\
&\inf_A \Bigl(\Tr^-  H_{A,V^0} -\int \Weyl^0_1 (x)\,dx +
\frac{1}{\kappa h^2}\int |\partial A|^2\,dx\Bigr)
\pm  Ch^{-\frac{4}{3}}.
\notag
\end{align}
Sure as $V$ (and surely $V^0$) are not sufficiently fast decaying at infinity the left (and for sure the right hand) expression in (\ref{4-10}) should be regularized as in section~\ref{sect-4}. However for potential decaying fast enough (faster than $|x|^{-2-\delta}$) regularization is not needed.

For $V^0$ we have an exact expression which concludes the proof.   
\end{proof}

\chapter{Several singularities}
\label{sect-5}

Consider now ``molecular'' case $M\ge 1$. Then we need more delicate arguments.
 
\section{Decoupling of singularities}
\label{sect-5-1}

Consider partition of unity $1=\sum_{0\le j\le m} \psi_j^2$ where $\psi_j$ is supported in $\frac{1}{3}a$-vicinity of $\x_j$ as $j=1,\ldots, m$ and $\psi_0=0$ in $\frac{1}{4}a$-vicinities of $\x_j$ (``near-nuclei'' and ``between-nuclei''partition elements).

\subsection{Estimate from above}
\label{sect-5-1-1}

Then
\begin{equation}
\Tr H^-_{A,V} =\sum_{0\le j\le m} \Tr  (\psi_j H^-_{A,V} \psi_j)
\label{5-1}
\end{equation}
and to estimate $\E^*$ from the above we impose an extra condition to $A$:
\begin{equation}
A=0 \qquad \text{as \ \ } \ell (x) \ge \frac{1}{5}a.
\label{5-2}
\end{equation}
Then in this framework we estimate 
\begin{equation}
|\Tr^- (\psi_0 H^-_{A,V} \psi_0)-\int \Weyl_1(x)\psi_0^2 (x)\,dx|\le
C h^{-1} a^{-\frac{1}{2}}.
\label{5-3}
\end{equation}
Proof is trivial by using $\ell$-admissible partition and applying results of the theory without any magnetic field.

So, to estimate $\E^*$ from above\footnote{\label{foot-3} Modulo error in (\ref{3-8}).} we just need to estimate from  above minimum with respect to $A$ satisfying (\ref{5-2}) of expression 
\begin{equation}
\Tr (\psi_j H^-_{A,V} \psi_j)-\int \Weyl_1(x)\psi_j^2 (x)\,dx + 
\frac{1}{\kappa h^2}  \int_{\{\ell_j(x)\le \frac{1}{5}a\}}  |\partial A|^2\,dx.
\label{5-4}
\end{equation}

\subsection{Estimate from below}
\label{sect-5-1-2}

In this case we use the same partition of unity $\{\psi_j^2\}_{j=0,1,\ldots, m}$ and estimate
\begin{gather}
\Tr H^- _{A,V}\ge \sum_{0\le j\le m} \Tr^- (\psi_j H_{A,V'} \psi_j)
\label{5-5}\\
\shortintertext{with}
V'= V+ 2h^2 \sum _j  (\partial \psi )^2
\label{5-6}\\
\shortintertext{and we also use decomposition} 
\int |\partial A|^2\,dx = \sum_{0\le j\le m} \int \omega_j^2 |\partial A|^2\, dx
\label{5-7}
\end{gather}
with
\begin{multline}
\omega_j(x) =1 \quad \text{as\ \ } \ell_j(x)\le \frac{1}{10}a,\qquad
\omega_j (x) \ge 1-C\varsigma  \quad \text{as\ \ } \ell_j (x)\le \frac{1}{2}a
\\ j=1,\ldots, m,
\label{5-8}
\end{multline}
\begin{equation}
\omega_0 \ge \epsilon_0 \varsigma\quad \text{as\ \ } \ell (x)\ge \frac{1}{5}a.
\label{5-9}
\end{equation}
So far $\varsigma>0$ is a constant but later it will be a small parameter.
Then since
\begin{multline}
\Tr^- (\psi_0 H_{A,V'} \psi_0)-\int \Weyl_1(x)\psi_0^2 (x)\,dx +
\frac{1}{\kappa h^2} \int \omega_0^2 |\partial A|^2\,dx\ge \\Ch^{-1}a^{-\frac{1}{2}}
\label{5-10}
\end{multline}
(again proven by partition) in virtue of \cite{MQT10} we are left with the estimates from below for
\begin{equation}
\Tr^- (\psi_j H_{A,V'} \psi_j)-\int \Weyl_1(x)\psi_j^2 (x)\,dx +
\frac{1}{\kappa h^2} \int \omega_j^2 |\partial A|^2\,dx.
\label{5-11}
\end{equation}

\begin{remark}\label{rem-5-1}
\begin{enumerate}[label=(\roman*), fullwidth]
\item 
Note that the error in $\Weyl_1$ when we replace $V'$ there by $V$ does not exceed  $Ch^{-1}(1+a^{-\frac{1}{2}})$ which is  less than error in (\ref{3-8}). Here we can also assume that $A$ satisfies (\ref{5-2}); we need just to replace $\varsigma$ by $\epsilon_0\varsigma$  in (\ref{5-8})--(\ref{5-9}). 

\item 
We also can further go down by replacing $\Tr^- (\psi_j H_{A,V'} \psi_j)$ by 
$\Tr  (\psi_j H^-_{A,V'} \psi_j)$.

\item
Therefore we basically have the same object for both estimates albeit with marginally different potentials ($V$ in the estimate from above and $V'$ in the estimate from below)  and with a weight $\omega_j^2$ satisfying (\ref{5-8})--(\ref{5-9}); in both cases $\omega =1$ as 
$\ell(x)\le \frac{1}{10}a$ but in the estimate from above $\omega (x)$ grows to $C_0$ and in the estimate from below $\omega(x)$ decays to $\varsigma$ as $\ell(x)\ge \frac{1}{3}a$ and in both cases condition (\ref{5-2}) could be imposed or skipped.

\item 
From now on we consider a single singularity at $0$ and we skip index $j$. However if there was a single singularity from the beginning, all arguments of this and forthcoming subsections would be unnecessary.
\end{enumerate}
\end{remark}

\subsection{Scaling}
\label{sect-5-1-3}

\begin{enumerate}[fullwidth, label =(\roman*)]
\item We are done as $z\asymp 1$ but as $z \ll 1$\,\footnote{\label{foot-4} As $z$ denotes $z_j$ we assume only that $z_1+\ldots+z_M\asymp 1$.} we need a bit more fixing. The problem is that $V \asymp z \ell ^{-1}$ only as $|x|\le z a$; otherwise $V\lesssim a^{-1}$ (where we assume that $a\le 1$). To deal with this we apply in the zone $\{x:\, z a \le |x|\le a\}$ the same procedure as before and its contribution to the error will be  $Ch^{-1} a^{-\frac{1}{2}}$ as $\rho = a^{-\frac{1}{2}}$ here. Actually we also need to keep $|x|\ge z^{-1} h^2$; so we assume that $z^{-1}h^2\le z a$ i.e.  $z  \ge a^{-\frac{1}{2}} h$.   

Now scaling $x\mapsto x'=x/ z a$, multiplying $H_{a,V}$ by $a$ (and therefore also multiplying $A$ by $a^{\frac{1}{2}}$, so $A\mapsto A'=a^{\frac{1}{2}}A$, $h\mapsto h'=h a^{-\frac{1}{2}} z^{-1}$; then the  magnetic energy becomes 
$\kappa^{-1}h^{-2}z\int \omega (x)^2 |\partial' A'|^2\, dx'$ where factors $a^{-1}$ and $az$ come from substitution $A=a^{-\frac{1}{2}}A'$ and scaling respectively. We need to multiply it by $a$ (as we multiplied an operator); plugging $h^{-2}=h'^{-2} a^{-1} z^{-2}$ we get the same expression as before but with $z'=1$, $a'=1$ and $h'=h a^{-\frac{1}{2}} z^{-1}\le 1$ and 
$\kappa'=\kappa z$ instead of $h$ and $\kappa$. If we establish here an error 
$O(h'^{-\frac{4}{3}})$ the final error will be $O(a^{-1} h'^{-\frac{4}{3}})= O(a^{-\frac{1}{3}}h^{-\frac{4}{3}}z^{\frac{4}{3}})$.

\item
 On the other hand, let $z  \le a^{-\frac{1}{2}} h$. Recall, we assume that $a\ge C_0 h^2$. Then  we can apply the same arguments as before but with $\bar{z}=a^{-\frac{1}{2}} h$ and we arrive to the same situation as before albeit with $h'=1$, $a'=1$, $\kappa'= \kappa a^{-\frac{1}{2}} h$ and with 
$z'= z/\bar{z}$. Then we have the trivial error estimate  $O(a^{-1})=O(a^{-\frac{1}{3}}h^{-\frac{4}{3}})$.
\end{enumerate}

\section{Main results}
\label{sect-5-2}

Combining results of the previous subsections with proposition~\ref{prop-2-7} we arrive to

\begin{theorem}\label{thm-5-2}
If $M\ge 2$, $\kappa \le \kappa^*$ and \textup{(\ref{2-30})} holds with $\nu>\frac{4}{3}$ then 
\begin{gather}
\E^* = \int \Weyl_1(x)\,dx + 2 \sum_j z_j^2 S(z_j \kappa) h^{-2}+ O(R_1+R_2)
\label{5-12}\\
\shortintertext{with}
R_1 =\left\{\begin{aligned}
&h^{-\frac{4}{3}}\qquad &&a\ge 1\\
&a^{-\frac{1}{3}}h^{-\frac{4}{3}}\qquad &&h^2\le a\le 1
\end{aligned}\right.\label{5-13}\\
\shortintertext{and}
R_2=\kappa h^{-2}
\left\{\begin{aligned}
&a^{-3}\qquad && a\ge |\log h|^{\frac{1}{3}},\\
&|\log h^2/a|^{-1} \qquad && h^2\le a \le |\log h|^{\frac{1}{3}}
\end{aligned}\right..
\label{5-14}
\end{gather}
\end{theorem}

\begin{proof}
To prove theorem we need to prove an estimate 
\begin{equation}
\frac{1}{\kappa h^2}\|\partial A \|^2 _{\{b\le \ell(x)\le 2b\}} 
\le CR_2
\label{5-15}
\end{equation}
where $R_*\le b\le a$ is a ``cut-off''. On the other hand we know that 
\begin{equation}
\frac{1}{\kappa^2 h^2}\|\partial A\|^2=-\frac{\partial S}{\partial \kappa}=O(1)
\label{5-16}
\end{equation}
and we need to recover the last factor in the definition of $R_2$.

As $a\ge 1$ we can have $a^{-3}$ because in virtue of (\ref{2-31}) the square of the partial norm in (\ref{5-16}) does not exceed $Ca^{-3}\kappa^2$. 

On the other hand, as $h^2\le R_* \le a$ we can select $b: R_*\le b\le a$ such that  he partial norm in (\ref{5-16}) does not exceed 
$C|\log (a/h^2) |^{-1}\cdot \|\partial A\|^2$. 
\end{proof}

\begin{remark}\label{rem-5-3}
\begin{enumerate}[label=(\roman*), fullwidth]
\item As $a\le |\log h|$  we do not need assumption (\ref{2-30});

\item In particular, as $a\ge 1$ and $\kappa \le a^3 h^{\frac{2}{3}}$ remainder estimate is $O(h^{-\frac{4}{3}})$.
\end{enumerate}
\end{remark}

\section{Problems and remarks}
\label{sect-5-3}

\begin{Problem}\label{Problem-5-4}
\begin{enumerate}[label=(\roman*), fullwidth]
\item As $\kappa \in [0, \kappa^*]$ with small enough $\kappa^*$ does $S(\kappa)$ really depend on $\kappa$ or $S(\kappa)=S(0)$? 
\item If $S(\kappa)$ really depends on $\kappa$, what is asymptotic behavior of $S(\kappa)-S(0)$ as $\kappa\to +0$: can one improve $S(\kappa)-S(0)=O(\kappa)$? 
\end{enumerate}
\end{Problem}
Any estimate better than $O(\kappa)$ would improve (with respect to $\kappa$) remainder estimates in theorems~\ref{thm-4-3} and~\ref{thm-5-2}.

\begin{Problem}\label{Problem-5-5}
Improve (as $a\ge 1$) estimates in theorem~\ref{thm-4-3} and~\ref{thm-5-2}   to those achieved in section~\ref{book_new-sect-12-5} of \cite{futurebook} for  $\kappa=0$ (i.e. without self-generated magnetic field). Namely there we were able to achieve $O(h^{-1})$ or even better, up to $O(h^{-1+\delta})$\,\footnote{\label{foot-5} Under global condition to Hamiltonian flow.}.

\begin{enumerate}[label=(\roman*), fullwidth]
\item The best outcome would be the same estimate $O(h^{-1})$ (or better~\footref{foot-5}) for all $\kappa \in [0,\kappa^*]$. 
\item Alternatively, we would like to see estimate 
$O( h^{-1}+\kappa^\mu h^{-\frac{4}{3}})$; in particular we would get estimate $O(1)$ for $\kappa =O(h^{2/(3\mu)})$  with exponent $\mu $ as large as possible.
\end{enumerate}
\end{Problem}

\input{MQT11.bbl}

\end{document}

%% file: defines.tex
\theoremstyle{plain}
\newtheorem{theorem}{Theorem}[chapter]

\newtheorem{proposition}[theorem]{Proposition}
\newtheorem{corollary}[theorem]{Corollary}

\theoremstyle{definition}

\newtheorem{Problem}[theorem]{Problem}
\theoremstyle{remark}
\newtheorem{remark}[theorem]{Remark}

\DeclareMathAlphabet{\mathpzc}{OT1}{pzc}{m}{it}

 \newcommand{\cX}{\mathcal{X}}

 \newcommand{\sC}{\mathscr{C}}

 \newcommand{\sH}{\mathscr{H}}

 \newcommand{\sL}{\mathscr{L}}

\newcommand{\orig}{{\mathsf{orig}}}

\newcommand{\Weyl}{{\mathsf{Weyl}}}

\newcommand{\E}{{\mathsf{E}}}

\newcommand{\x}{{\mathsf{x}}}

\newcommand{\bC}{{\mathbb{C}}}

\newcommand{\bR}{{\mathbb{R}}}

\def\1{\boldsymbol {|}}
\newcommand{\boldsigma}{{\boldsymbol{\sigma}}}

\newcommand{\Def}{\mathrel{\mathop:}=}

\newcommand{\supp}{\operatorname{supp}}

\newcommand{\tr}{\operatorname{tr}}
\newcommand{\Tr}{\operatorname{Tr}}

\newenvironment{claim*}[1][{}]{\vglue10pt
\begin{trivlist}
\item[{\hskip\labelsep#1}]}{\vglue10pt\end{trivlist}}

\newcounter{note}

\DeclareTextCommand{\textinfty}{PU}{\9042\036}

\DeclareTextCommand{\textge}{PU}{\9042\145}
\DeclareTextCommand{\textle}{PU}{\9042\144}
\DeclareTextCommand{\texthat}{PD1}{\136}

%% file: MQT11.bbl
\bibliographystyle{alpha}

